\documentclass[11pt,leqno]{amsart}%
\usepackage{amsmath}
\usepackage{amsfonts}
\usepackage{amssymb}
\usepackage{dsfont}
\usepackage{graphicx}
\usepackage[margin=1.90cm]{geometry}
\providecommand{\U}[1]{\protect\rule{.1in}{.1in}}
\newtheorem{theorem}{Theorem}[section]
\newtheorem*{acknowledgement*}{Acknowledgements}

\newtheorem{corollary}[theorem]{Corollary}

\newtheorem{definition}[theorem]{Definition}

\newtheorem{lemma}[theorem]{Lemma}

\newtheorem{proposition}[theorem]{Proposition}
\newtheorem{remark}[theorem]{Remark}

\def\RRR{{\mathrm R}}

\def\Ric{{\mathrm{Ric}}}

\def\vol{{\mathrm {vol}}}
\def\e{{\varepsilon}}

\begin{document}
\title[Extremals of Log Sobolev inequality]{Extremals of Log Sobolev inequality on non-compact manifolds and Ricci soliton structures}

\subjclass[2010]{53C21, 53C44, 35J20}
\keywords{Log Sobolev inequality, extremals, Ricci solitons}

\date{\today}

\author {Michele Rimoldi}
\address[Michele Rimoldi]{Dipartimento di Scienze Matematiche "Giuseppe Luigi Lagrange", Politecnico di Torino, Corso Duca degli Abruzzi 24, I-10129 Torino, Italy}
\email{michele.rimoldi@polito.it}
\author {Giona Veronelli}
\address[Giona Veronelli]{Dipartimento di Matematica e Applicazioni, Universit\`a di Milano Bicocca, via R. Cozzi 53, I-20126 Milano, Italy}
\email{giona.veronelli@unimib.it}

\begin{abstract}
In this paper we establish the existence of extremals for the Log Sobolev functional on complete non-compact manifolds with Ricci curvature bounded from below and strictly positive injectivity radius, under a condition near infinity. This extends a previous result by Q. Zhang where a $C^1$ bound on the whole Riemann tensor was assumed. When Ricci curvature is also bounded from above we get exponential decay at infinity of the extremals. As a consequence of these analytical results we establish, under the same assumptions, that non-trivial shrinking Ricci solitons support a gradient Ricci soliton structure. On the way, we prove two results of independent interest: the existence of a distance-like function with uniformly controlled gradient and Hessian on complete non-compact manifolds with bounded Ricci curvature and strictly positive injectivity radius and a general growth estimate for the norm of the soliton vector field. This latter is based on a new Toponogov type lemma for manifolds with bounded  Ricci curvature, and represents the first known growth estimate for the whole norm of the soliton field in the non-gradient case.
\end{abstract}

\maketitle
\tableofcontents

\section{Introduction and main results}

A Ricci soliton structure on a Riemannian manifold $(M^m, g)$ is the choice of a vector field $X$ (if any) such that 
\begin{equation}\label{RS}
\mathrm{Ric}+\frac{1}{2}\mathcal{L}_Xg=\lambda_{S} g,
\end{equation}
for some constant $\lambda_S\in\mathbb{R}$. The soliton is called expanding, steady or shrinking if, respectively, $\lambda_S<0$, $\lambda_S=0$ or $\lambda_S>0$. When $X=\nabla f$ , for some $f\in C^{\infty}(M)$, we say that the Ricci soliton is gradient.

It is well known that expanding and steady compact Ricci solitons are Einstein; \cite{ivey}. Using the existence of extremals of the $\mathcal{W}$-entropy and its monotonicity formula, G. Perelman proved that every non-trivial compact shrinking Ricci soliton supports a shrinking gradient Ricci soliton structure; \cite{P1} (see also \cite{ELNM} for a completely elliptic proof of this result, not involving Perelman's monotonicity formula). This result was later extended by A. Naber, \cite{N}, to complete non-compact shrinking Ricci solitons with bounded curvature tensor. However, rather than using extremals of the $\mathcal{W}$-entropy he used a careful analysis of the reduced length function and Ricci flow's convergence techniques.

About the converse, it is important to notice that J. Carrillo and L. Ni, \cite{CN}, proved that potential functions on gradient shrinking Ricci solitons provide extremals for the $\mathcal{W}$-entropy also in the complete non-compact case without any curvature assumption. Hence, the problem of finding an extremal for the $\mathcal{W}$-entropy and the problem of finding a gradient Ricci soliton structure on a manifold supporting a complete shrinking Ricci soliton structure are equivalent modulo justification for an integration by parts, as one can see from the proof in \cite{ELNM}. A natural question that arises, and which is the initial motivation of this paper, is if the elliptic technique coming from \cite{P1}, \cite{ELNM} can be extended to guarantee the existence of a gradient Ricci soliton structure on a complete non-compact shrinking Ricci soliton considering different curvature conditions from those considered in \cite{N}.

An essential step in this program is to study the existence of extremals of the $\mathcal{W}$-entropy. This functional turns out to be intimately tied with the usual Log Sobolev functional which was extensively studied in literature. In particular, the existence problem for extremal functions of the Log Sobolev functional in the compact case was solved by O. Rothaus, \cite{Ro}. On the other hand, in the complete non-compact case, in \cite[Remark 3.2]{P1} Perelman raised the question whether extremals for the Log Sobolev functional exist. A motivation for studying the complete non-compact case is the fact that many interesting singularity models of the Ricci flow are non-compact, the cylinder $\mathbb{S}^2\times\mathbb{R}$ being the most basic example.

Recently there have been some progresses on this problem in \cite{ZHQ}, where Q. S. Zhang provided a result on the existence of extremals on complete non-compact manifolds with bounded geometry satisfying an extra condition at infinity which avoids the escape at infinity of the entropy content of a minimizing sequence. Here bounded geometry means that the Riemann curvature tensor and all its covariant derivatives are bounded and that there is a uniform positive lower bound on the volume of geodesic balls of radius $1$. It is well known that these conditions imply a positive lower bound on the injectivity radii on the manifold; \cite{CheegerPhD}. Furthermore, in the same paper, Zhang also proved that an extremal function may not exist if the condition at infinity is violated.
\medskip

The main analytical result of this paper establishes the existence of extremals for the Log Sobolev functional replacing the bounded geometry condition considered in \cite{ZHQ} with only a lower control on the Ricci curvature and a positive lower bound on the injectivity radius. Moroever, asking bounded Ricci curvature, we can conclude that the extremal function decays more than exponentially at infinity.

For the detailed definitions of the functional $\mathcal{L}$, the best Log Sobolev constant $\lambda$, and the best Log Sobolev constant at infinity $\lambda_{\infty}$, we refer, respectively, to \eqref{LS_funct}, \eqref{bls_const}, and \eqref{bls_constinf} in Section \ref{DefNot} below.

\begin{theorem}\label{ExtrLSI}
Let $(M^m, g)$ be a (connected) complete non-compact Riemannian manifold and suppose that
\begin{equation}\label{Our}
\mathrm{Ric}\geq -(m-1)K\qquad\mathrm{and}\qquad\mathrm{inj}_{(M,g)}\geq i_{0}>0,
\end{equation}
for some $K\in[0,+\infty)$, $i_{0}\in\mathbb{R}^{+}$. If $\lambda<\lambda_{\infty}$, then there exists a smooth extremal $v$ for the Log Sobolev functional $\mathcal{L}$.
In addition, if instead of the bound $\mathrm{Ric}\geq-(m-1)K$ we assume that
\begin{equation*}
|\mathrm{Ric}|\leq(m-1)K,
\end{equation*}
then, having fixed a point $o\in M$, there exist positive constants $C, c>0$ such that the extremal $v$ satisfies
\begin{equation}\label{ExpDec}
v(x)\leq Ce^{-cd^{2}(x,o)},
\end{equation}
for any $x\in M$.
\end{theorem}

\begin{remark}
\textbf{(a)} Even though the Log Sobolev functional $\mathcal{L}$ we are dealing with contains the scalar curvature $R$, the result still holds if one deletes the term containing the scalar curvature. The proof requires only minor adjustement. \\
\textbf{(b)} For comments on the generality of the condition at infinity $\lambda<\lambda_{\infty}$ we refer to \cite{ZHQ}, where also some examples of Riemannian manifolds satisfying this condition are constructed.
\end{remark}

As we said above, the conclusion of Theorem \ref{ExtrLSI} was obtained by Zhang under the stronger assumption of bounded geometry (in the sense explained above). Basically, he needs this assumption to ensure:
\begin{itemize}
\item [(i)] the validity of a Sobolev inequality and of Bishop-Gromov comparison theorem;
\item[(ii)] the existence of a distance-like function with uniformly controlled gradient and Hessian, which in turn is used to get an a-priori decay of the type \eqref{ExpDec} for subsolutions of the Euler-Lagrange equation relative to the Log Sobolev functional;
\item[(iii)] Hamilton's version of Cheeger-Gromov compactness theorem for pointed Riemannian manifolds.
\end{itemize}
In fact, in his proof he constructs a minimizing sequence of subsolutions $\{v_k\}$ to the Euler-Lagrange equation for $\mathcal{L}$, each of them satisfying $v_k(x_k)>\delta>0$ at some point $x_k\in M$. In case the sequence $\{x_k\}$ is unbounded, one can apply Cheeger-Gromov convergence to $\{(M, g, x_k)\}$ and obtain the existence of a subsolution of the Euler-Lagrange equation on the limit manifold. This is showed to be in contradiction with the condition $\lambda<\lambda_\infty$. Hence $\{x_k\}$ is necessarily bounded. In this second case classical arguments give the existence of a smooth extremal. To treat both the bounded and unbounded case he uses the exponential decay alluded to in point (ii); see \cite[Lemma 2.3]{ZHQ}. 

It is well known that our assumption \eqref{Our} is sufficient to guarantee both Sobolev inequality and Bishop-Gromov comparison. As a matter of fact, up to asking also an upper bound on $\mathrm{Ric}$, one can also get the validity of (ii). This follows from the following result, which apparently has never been observed before in literature in this generality. Under stronger assumptions this was obtained by L.-F. Tam in \cite{T}. 


\begin{proposition}\label{DistLike}
Given $m\geq 2$, $K\in\left[0,\infty\right)$, there exists a constant $C_{m,K}\in\left(1,\infty\right)$, depending only on $m$ and $K$, such that if $(M^{m}, g)$ is a complete non-compact Riemannian manifold with $|\mathrm{Ric}|\leq (m-1) K$ and $\mathrm{inj}_{(M,g)}\geq i_{0}>0$, $o\in M$ and $r(x):=d(x,o)$, then there exists $h\in C^{\infty}(M)$ such that
\begin{eqnarray}
	&r(x)+1\leq h(x)\leq r(x)+C_{m,K}\label{A}\\
	&|\nabla h|(x)\leq C_{m,K}\label{B}\\
	&|\mathrm{Hess}(h)|(x)\leq C_{m,K}.\label{C}
\end{eqnarray} 
\end{proposition}

However, in point (iii) the bounds on the full curvature tensor seem unavoidable. To circumvent this problem, we propose a different strategy to prove the existence of the extremal. By the way, our proof does not require an a-priori exponential decay for subsolutions (and hence an upper bound for $\mathrm{Ric}$). The validity of \eqref{ExpDec} is nevertheless proved a-posteriori, since it will be exploited in the proof of the geometric result.

Our proof goes as follows. As in \cite{ZHQ}, we consider a minimizing sequence $\{v_k\}$, bounded in $W^{1,2}$, made up of subsolutions of the Euler-Lagrange equation. The assumption $\lambda<\lambda_\infty$ is then used to guarantee that the weak limit $v$ is not null. Then we prove that $v$ is a weak solution and standard regularity theory applies. All along the proof, the equation is used repeatedly to deal with the logarithmic term.\\
   
Exploiting Theorem \ref{ExtrLSI} we are able to obtain the following geometric consequence.

\begin{theorem}\label{th_RicSol}
Let $(M^m, g)$ be a connected complete non-compact Riemannian manifold which supports a shrinking Ricci soliton structure. Suppose that there exist positive constants $K$ and $i_0$ such that
\begin{equation*}
|\mathrm{Ric}|\leq(m-1)K,\qquad\mathrm{and}\qquad\mathrm{inj}_{(M,g)}\geq i_{0}>0
\end{equation*}
and that $\lambda<\lambda_{\infty}$. Then $(M^m, g)$ supports also a shrinking gradient Ricci soliton structure.
\end{theorem}
\medskip

In order to prove Theorem \ref{th_RicSol} one would like to mimic the computations done in the compact case in \cite{ELNM}. Since we are in the non-compact setting, aiming to justify the integration by parts (see \eqref{Tv}) using a Stokes' theorem \`a la Gaffney-Karp we need a control on the growth at infinity of the soliton field $X$. 

In this regard, it is well known that the growth of the radial part of $X$ can be controlled when the Ricci curvature is bounded. Indeed if $\Ric$ is bounded from below, it is not difficult to see that $\left\langle X,\nabla r\right\rangle$ has at most linear growth. Moreover, it was proven by A. Naber, \cite{N} that if $\Ric$ is bounded from above then $\left\langle X,\nabla r\right\rangle$ grows at least linearly. Here $r$ is the distance function from a fixed reference origin. Note also that for gradient shrinking Ricci solitons it was proven in \cite{Z} that the whole $|X|$ grows at most linearly. To the best of our knowledge a growth estimate on the whole $|X|$ in the non-gradient case is not known so far. Actually in this setting we are not able to get such a linear bound for $|X|$, yet we can prove that the field can not grow much more than exponentially, which is in fact enough to our purposes. Note that the following result also concerns steady and expanding Ricci solitons.

\begin{theorem}\label{th_SolGrowth}
Let $(M^m,g)$ be a complete non-compact $m$-dimensional Ricci soliton satisfying \eqref{RS} for some $X\in\mathcal{X}(M)$ and $\lambda_{S}\in\mathbb{R}$. Suppose that $\left|\Ric\right| \leq (m-1)K$ for some constant $K\geq 0$. For any reference point $o\in M$ there exists a positive constant $C>0$, depending on $m$, $K$, $\lambda_S$, $\vol(B_1(o))$ and on
\begin{equation*}
X^\ast:=\max_{y\in\overline{B_1(o)}}|X(y)|,
\end{equation*}
such that for all $q\in M$ it holds
\begin{equation*}
|X|(q)\leq \begin{cases}
C d(q,o)^{m}&\,\textrm{if}\,\,K=0\\
C d(q,o)e^{(m-1)\sqrt Kd(q,o)}&\ \textrm{if}\,\,K>0.
\end{cases}
\end{equation*}
\end{theorem}
This result relies on a Toponogov's type estimate of independent interest, which we call Ricci Hinge Lemma. Beyond the proof of Theorem \ref{th_SolGrowth}, this estimate applies more generally to control the growth of any vector field $X$ along which one can control $\mathcal{L}_Xg$, such as for instance Killing vector fields; see Corollary \ref{coro_kill} below. We are not aware of previous results in this direction.

The paper is organized as follows. In Section \ref{DefNot} we introduce some notation and basic definitions. Section \ref{PrelRes} and Section \ref{ThExtr1} are devoted to the proof of the first part of Theorem \ref{ExtrLSI} that is the existence of the extremal. In Section \ref{AppA} we present a proof of the construction of distance-like functions on complete non-compact manifolds with bounded Ricci curvature and a control on the injectivity radius, and we deduce the existence on these manifolds of Hessian cut-off functions. From these results the second part of Theorem \ref{ExtrLSI} immediately follows, as shown in Section \ref{ThExtrII}. We end the paper with Section \ref{GRSRS} which finally deals with Ricci soliton structures. In a first part we prove the Ricci Hinge Lemma, from which the general result on the growth of the soliton field on manifolds with bounded Ricci curvature can be deduced, while in the second part we deal with the integration by parts which permits to conclude the proof of Theorem \ref{th_RicSol} .

\section{Basic definitions and notation}\label{DefNot}

Let $(M^m, g)$, be a complete connected Riemannian manifold of dimension $m\geq3$ and denote by $R$ its scalar curvature and by $d\rm{vol}$ the Riemannian volume measure.\\

The \textbf{$\mathcal{W}$-entropy} is defined for $f\in W^{1,2}(M, e^{-f}d\rm{vol})$ as
\[
\ \mathcal{W}(g,f,\tau):=\int_M\left[\tau(\left|\nabla f\right|^2 +R)+f-m\right](4\pi\tau)^{-\frac{m}{2}}e^{-f}d\rm{vol},
\]
where $\tau>0$ is a scale parameter and we ask $R\in L^{1}(M, e^{-f}d\mathrm{vol})$. 
\medskip

Setting $v^2=(4\pi\tau)^{-\frac{m}{2}}e^{-f}$, we may rewrite the $\mathcal{W}$-entropy for $v\in W^{1,2}(M)$ as follows  
\[
\ \mathcal{K}(g,v,\tau)=\int_M \left(\tau(4|\nabla v|^2+R v^2)-v^2\ln v^2-mv^2 -\frac{m}{2}\ln(4\pi\tau)v^2\right)d\mathrm{vol}.
\]
The problem of minimizing the $\mathcal{W}$-entropy under the constraint
\[
\ f\in \left\{f\in C_{c}^{\infty}(M) \,\,\mathrm{s.t.}\, \int_M(4\pi\tau)^{-\frac{m}{2}}e^{-f}d\mathrm{vol}=1\right\}
\] 
is hence equivalent to the problem of minimizing $\mathcal{K}(g, v, \tau)$ under the constraint
\[
\ v\in\mathcal{U}=\left\{v\in C_{c}^{\infty}(M) \,\,\mathrm{s.t.}\, \int_M v^2d\mathrm{vol}=1 \right\}.
\] 
\medskip
Let $c>0$ be a positive constant, then the functional $\mathcal{K}$ has the following scale invariance property
\[
\ \mathcal{K}\left(cg,c^{-\frac{m}{4}}v, c\tau\right)=\mathcal{K}(g,v,\tau).
\]
Hence we can restrict the study, without loss of generality, to the case $\tau=1$. Note that, if $\left\|v\right\|_{L^{2}(M)}=1$, then
\begin{eqnarray*}
\ \mathcal{K}(g,v,1)&=&\int_{M}\left[(4|\nabla v|^2+Rv^{2})-v^{2}\ln(v^2)\right]d\mathrm{vol}-\frac{m}{2}(\ln 4\pi)-m\\
&=&\mathcal{L}(v,g)-\frac{m}{2}(\ln 4\pi)-m.
\end{eqnarray*}
Here $\mathcal{L}(v,g)=\mathcal{L}(v,M,g)$ is the \textbf{Log Sobolev functional} on $(M,g)$ perturbed by the scalar curvature of the manifold, which is defined for $v\in W^{1,2}(M)$ as
\begin{equation}\label{LS_funct}
\ \mathcal{L}(v,M, g):=\int_{M}\left(4|\nabla v|^2+Rv^2-v^{2}\ln v^{2}\right)d\mathrm{vol}.
\end{equation}
We define the \textbf{best Log Sobolev constant of a domain $\Omega\subset M$} as
\begin{equation}\label{bls_const}
\ \lambda(\Omega)=\inf\left\{\int_{\Omega}\left[4|\nabla v|^2+Rv^{2}-v^{2}\ln v^{2}\right]d\mathrm{vol}\,\,\mathrm{s.t.}\,\,v\in C_{c}^{\infty}(\Omega); \,\left\|v\right\|_{L^{2}(\Omega)=1}\right\}.
\end{equation}
When $\Omega=M$, we will denote by $\lambda:=\lambda(M)$ the \textbf{best Log Sobolev constant of $(M, g)$}. \\
The \textbf{best Log Sobolev constant at infinity of $(M, g)$} is the quantity
\begin{equation}\label{bls_constinf}
\lambda_{\infty}:=\liminf_{r\to\infty}\lambda(M\setminus B_{r}(o)).
\end{equation}
\medskip

We now give the following
\begin{definition}
Suppose that $\lambda>-\infty$. A function $v\in W^{1,2}(M)$ is called an \textbf{extremal of the Log-Sobolev functional $\mathcal{L}$} on $(M,g)$, if $\left\|v\right\|_{L^{2}(M)}=1$ and
\[
\ \int_{M}\left(4|\nabla v|^2+Rv^{2}-v^{2}\ln v^{2}\right)d\mathrm{vol}=\lambda.
\]
\end{definition}

It is worthwhile to note that, according to this definition, extremals need not to belong to $C_{c}^{\infty}(M)$, and hence to the class of functions $\mathcal U$ we are minimizing in.

The \textbf{Euler-Lagrange equation} for the Log Sobolev functional $\mathcal{L}$ is given by
\begin{equation}\label{EL}
4\Delta v-Rv+2v\ln v+\lambda v =0 
\end{equation}

Here, and from the point onward, we implicitly assume that $v\geq 0$ when $\ln v$ appears and that $v\ln v(x)=0$ when $v(x)=0$.

\section{Some preliminary results}\label{PrelRes}

An analysis of the proof of Lemma 2.1 in \cite{ZHQ} gives the validity of the following 
\begin{lemma}\label{2.1}
Let $(M^{m}, g)$ be a complete Riemannian  manifold such that
\begin{equation*}
\mathrm{Ric}\geq -(m-1)K\qquad\mathrm{and}\qquad\mathrm{inj}_{(M,g)}\geq i_{0}>0.
\end{equation*}
Let $x\in M$ and suppose $v$ is a bounded solution to \eqref{EL} in the ball $B_2(x)\subset M$ such that $\left\|v\right\|_{L^{2}(B_2(x))}\leq 1$. Then the following mean value type inequalities hold.
\begin{itemize}
	\item[(a)] 
There exists a positive constant $C=C(m,K, i_{0}, \lambda)$ such that
\[
\ \sup_{B_1(x)}v^{2}\leq C\int_{B_2(x)}v^{2}d\mathrm{vol}.
\]
\item[(b)] There exists a positive constant $C=C(m,K, i_{0}, \lambda, \sup_{B_{1}(x)}|\nabla R|)$ such that
\[
\ \sup_{B_{1/2}(x)}|\nabla v|^{2}\leq C\int_{B_1(x)}v^{2}d\mathrm{vol}.
\]
\end{itemize}
\end{lemma}

The content of the following lemma is that in our assumptions the variational problem for the Log Sobolev functional $\mathcal{L}$ is well defined. The proof is standard and follows from a combination of Sobolev's and Jensen's inequality. We present it here for the sake of completeness.
 
\begin{lemma}
Let $(M^{m}, g)$ be a complete Riemannian  manifold such that
\begin{equation*}
\mathrm{Ric}\geq -(m-1)K\qquad\mathrm{and}\qquad\mathrm{inj}_{(M,g)}\geq i_{0}>0.
\end{equation*}
Then $\lambda$ is finite.
\end{lemma}

\begin{proof}
Under our assumptions, it is well-known that the following Sobolev inequality holds for all $u\in C_{c}^{\infty}(M)$
\begin{equation}\label{Sob}
\int_M u^{2^*}d\mathrm{vol}\leq \left(C_M\int_M\left(|\nabla u|^2+ u^2\right)\,d\mathrm{vol}\right)^{\frac{m}{m-2}},
\end{equation}
for some constant $C_M$ depending only on $(M, g)$.

Since $\ln$ is concave and $\left\|u\right\|_{L^{2}(M)}=1$, by Jensen's inequality we obtain
\begin{equation*}
\int_M u^2 \ln u^2d\mathrm{vol}\leq\frac{m-2}{2}\ln\left(\int_M u^{2^{*}}\,d\mathrm{vol}\right).
\end{equation*}
Using \eqref{Sob}, we thus get that
\begin{equation*}
\ln \left(\int_M u^{2^*}d\mathrm{vol}\right)\leq\frac{m}{m-2}\ln\left(C_M\int_M\left(|\nabla u|^2+ u^2\right)d\mathrm{vol}\right),
\end{equation*}
from which
\begin{equation}\label{Rot1}
\int_Mu^2\ln u^2d\mathrm{vol}\leq \frac{m}{2}\ln\left(C_M\int_M\left(|\nabla u|^2+u^2\right)d\mathrm{vol}\right).
\end{equation}
Using \eqref{Rot1} we observe now that
\begin{eqnarray*}
&&\,\,\,\,\,\int_M \left(4|\nabla u|^2-u^2\ln u^2\right)d\mathrm{vol}\\&&\geq \int_M4 (|\nabla u|^2+u^2)d\mathrm{vol} -4\int_Mu^2d\mathrm{vol}-\frac{m}{2}\ln\left(C_M\int_M(|\nabla u|^2+u^2)d\mathrm{vol}\right)\\
&&\geq 4\left[\int_M(|\nabla u|^2+u^2)d\mathrm{vol}-1\right]-\frac{m}{2}\ln \left(C_M\int_M(|\nabla u|^2+u^2)d\mathrm{vol}\right)\\
&&=\Phi(\left\|u\right\|_{W^{1,2}(M)}), 
\end{eqnarray*}
where $\Phi:\mathbb{R}\to\mathbb{R}$ is defined by $\Phi(t)=4(t^2-1)-\frac{m}{2}\ln (C_M t^2)$. Since $\Phi$ is bounded from below, we obtain that for any $u\in \mathcal{U}$
\[
\mathcal{L}(v, M, g)\geq  \inf_{M}R+ C,
\]
for some constant $C\in \mathbb{R}$, and hence $\lambda$ is finite.
\end{proof}

\section{Proof of Theorem \ref{ExtrLSI}: first part}\label{ThExtr1}

In this section we present the proof of the first assertion of Theorem \ref{ExtrLSI}, that is the existence of the extremal.
\medskip

By Proposition 2.1 in \cite{GW} we know that, given $o\in M$, there exists a $L\in C^{\infty}(M)$ with
\[
\ |L(x)-d(x, o)|\leq 1\qquad\mathrm{and}\qquad|\nabla L|(x)\leq 2\quad\mathrm{on}\quad M.
\] 
For any positive integer $k$, consider the domain
\[
\ D(o, k)=\left\{x\in M \,\mathrm{s.t.}\, L(x)<k\right\}
\]
and let $\lambda_{k}:=\lambda(D(o,k))$. Note that we can choose $L$ in such a way that $\partial D(o,k)$ is smooth for any $k\in \mathbb N$.

According to \cite{Ro}, $\lambda_{k}$ is finite and, for every $k$, there exists a  non-negative extremal on $D(o,k)$ in $W^{1,2}(D(o,k))\cap C^0(\overline{D(o,k)})$ with $\left\|v_{k}\right\|_{L^{2}(D(o,k))}=1$ which satisfies
\begin{equation*}
\begin{cases}
4\Delta v_{k}-Rv_{k}+2v_{k}\ln v_{k}+\lambda_{k}v_{k}=0, &\mathrm{in}\quad D(o, k)\\
v_{k}=0,&\mathrm{on}\quad\partial D(o,k)
\end{cases}
\end{equation*}
We extend $v_{k}$ to $M$ by setting it equal to $0$ on $M\setminus D(o,k)$. Hence $v_{k}\in C^{\infty}(\mathrm{int}(D(o,k)))\cap C^{0}(\overline{D(o,k)})$, $v_{k}\in W^{1, 2}(M)$ with $\left\|v_{k}\right\|_{W^{1,2}(M)}=\left\|v_{k}\right\|_{W^{1,2}(D(o,k))}$, and $\left\|v_{k}\right\|_{L^{2}(M)}=1$. Note also that, by definition, $\mathcal{L}(v_{k}, M, g)=\lambda_{k}\searrow\lambda$.

We claim that, under our assumptions, $\left\{v_{k}\right\}$ is uniformly bounded in $W^{1,2}(M)$.

Indeed, by Jensen's and Sobolev's inequalities we have that
\[
\ \int_{M}v_{k}^{2}\ln{v_{k}^{2}}d\mathrm{vol}\leq \frac{m}{2}\ln\left(A\int_{M}|\nabla v_k|^2d\mathrm{vol}+B\right),
\] 
for some constants $A, B\,\in\mathbb{R}$. Hence, since for every $\sigma\geq 0$, $t>0$, it holds that $\ln(t)\leq \sigma t -1 -\ln\sigma$, we deduce
\begin{equation*}
\int_{M}v_{k}^{2}\ln v_{k}^2d\mathrm{vol}\leq\frac{m}{2}\,\sigma A\int_{M}|\nabla v_k|^2d\mathrm{vol}+\frac{m}{2}\,\sigma B-\frac{m}{2}-\frac{m}{2}\ln \sigma.
\end{equation*}
Choosing $\sigma$ small enough we get that
\begin{align*}
\lambda_{k}=&\mathcal{L}(v_{k}, M, g)=\int_{M}\left(4|\nabla v_{k}|+Rv_{k}^2-v_{k}^2\ln v_{k}^2\right)d\mathrm{vol}\\
\geq&\int_{M}4|\nabla v_{k}|^{2}d\mathrm{vol}+\int_{M}\inf_{M}Rv_{k}^{2}d\mathrm{vol}-\frac{m}{2}\,\sigma A\int_{M}|\nabla v_k|^2d\mathrm{vol}-\frac{m}{2}\,\sigma B+ \frac{m}{2}+\frac{m}{2}\ln\sigma\\
\geq&\left(4-\frac{m}{2}\,\sigma A\right)\int_{M}|\nabla v_{k}|^{2}d\mathrm{vol}+C(\inf_{M}R, m, B, \sigma)\\
\geq&\int_{M}|\nabla v_{k}|^{2}d\mathrm{vol}+C(\inf_{M}R, m, A, B).
\end{align*}
Hence, for some constant C,
\begin{equation*}
\int_{M}|\nabla v_{k}|^{2}d\mathrm{vol}\leq \lambda_{k}+C\leq 2|\lambda| +C,
\end{equation*}
for $k\gg 1$, as claimed.
\medskip

 Therefore, up to passing to a subsequence, there exists $v\in W^{1,2}(M)$ such that $v_{k}\to v$ weakly in $W^{1,2}(M)$. As a standard consequence, by lower semicontinuity of the $W^{1,2}$-norm, we hence have that
\[
\ \left\|v\right\|_{W^{1,2}(M)}\leq \liminf_{k\to\infty}\left\|v_{k}\right\|_{W^{1,2}(M)}.
\]
Moreover, by the uniformly boundedness and the Rellich-Kondrachov compactness theorem, $v_{k}\to v$ strongly in $L^{p}$ on compact sets for every $p\in\left(1,2^{*}\right)$ and a.e. in $M$.
In particular $v\in W^{1,2}(M)$, $v\geq 0$ a.e. in $M$ and, since $\int_{A}v_k^{2}d\mathrm{vol}\leq 1$ for every $A\subset M$ compact, the same holds for $v$, and hence $\int_{M}v^{2}d\mathrm{vol}\leq 1$. Actually, we will show in Lemma \ref{lem4.3} below that $\int_{M}v^2d\mathrm{vol}=1$.

\begin{lemma}
$v$ is strictly positive on a set of positive measure in $M$.
\end{lemma}

\begin{proof}
We reason by contradiction, assuming that $v\equiv 0$ a.e. in $M$, and get a contradiction to the assumption that $\lambda<\lambda_{\infty}$. In the following, integrals are meant with respect to $d\mathrm{vol}$ unless otherwise specified.

Consider an exhaustion $E_{i}\nearrow M$ consisting of relatively compact domains with smooth boundary. Then for every $i$ fixed, there exists a $K_{i}$ such that the tubular neighborhood $B_{2}(E_{i})\subset D(o,k)$ for every $k\geq K_{i}$. We hence consider the sequence $\left\{v_{k}\right\}_{k\geq K_{i}}$ on $E_{i}$.

We need first the following 
\medskip

\textsc{Step 1. }If $v\equiv 0$ a.e. in $M$, then, for every $E_{i}$, $v_{k}\to 0$ in $C^{1}(E_{i})$ as $k\to\infty$.
\medskip

We have that $\int_{D(o, K_{i})}v_{k}^2\,d\mathrm{vol}\to 0$ as $k\to\infty$ since $D(o, K_{i})$ is compact. 
By Lemma \ref{2.1} applied to the $v_{k}$'s we have that there exists a constant $C= C(m, K, i_{0}, \lambda, \sup_{D(o, K_{i})}|\nabla R|)$ such that, for all $x\in E_{i}$,
\begin{align}
v_{k}^{2}(x)\leq&\sup_{B_1(x)}v_{k}^{2}\leq C\int_{B_2(x)}v_{k}^2\,d\mathrm{vol}\leq C\int_{D(o, K_{i})}v_{k}^{2}\,d\mathrm{vol}\nonumber\\
|\nabla v_{k}|^2(x)\leq& \sup_{B_{1/2}(x)}|\nabla v_{k}|^2\leq C\int_{B_1(x)}v_{k}^2\,d\mathrm{vol}\leq C\int_{D(o, K_{i})}v_{k}^{2}\,d\mathrm{vol}\label{Lemma3.1_2}.
\end{align} 
Letting $\varepsilon_{k}:=C\int_{D(o, K_{i})}v_{k}^{2}\,d\mathrm{vol}$, then $\varepsilon_{k}\to 0$ as $k\to\infty$, by our assumption. 
\medskip

\textsc{Step 2.} $\forall \, i$ and $\forall\,\varepsilon>0$ there exists $u\in C^{\infty}(M\setminus E_{i-1})\cap W^{1,2}(M)$ such that $\mathrm{supp}\,u\subset\subset M\setminus E_{i-1}$ and $\mathcal{L}(u, M, g)<\lambda + \varepsilon$. 
\medskip

Fix $i$ and $\varepsilon>0$. Let $\varphi\in C^{\infty}(M)$, $0\leq\varphi\leq 1$ such that $|\nabla \varphi|$ is bounded and
\[
\ \varphi_{|M\setminus E_{i}}=1,\qquad\varphi_{|E_{i-1}}=0.
\]
For $k\geq K_{i}$ define $u_{k}:=v_{k}\varphi$. Then $u_{k}\in C^{\infty}(E_{i})\cap W^{1,2}(M)$ and
\begin{align*}
\mathcal{L}(u_{k}, M, g)=&\mathcal{L}(v_{k}, M\setminus E_{i}, g)+\int_{E_{i}}4|\nabla \varphi|^2v_{k}^2\,d\mathrm{vol}+\int_{E_{i}}8\left\langle \nabla \varphi, \nabla v_{k}\right\rangle \,d\mathrm{vol}+\int_{E_{i}}4\varphi^{2}|\nabla v_{k}|^2\,d\mathrm{vol}\\
&+\int_{E_{i}}R\varphi^2v_{k}^{2}\,d\mathrm{vol}-\int_{E_{i}}\varphi^2v_{k}^2\ln\varphi^{2}\,d\mathrm{vol}-\int_{E_{i}} \varphi^2v_{k}^2\ln v_{k}^2\,d\mathrm{vol}\\
=&\mathcal{L}(v_{k}, M\setminus E_{i}, g)+\int_{E_{i}\setminus E_{i-1}}4|\nabla \varphi|^2v_{k}^2\,d\mathrm{vol}+\int_{E_{i}\setminus E_{i-1}}8\left\langle \nabla \varphi, \nabla v_{k}\right\rangle \,d\mathrm{vol}+\int_{E_{i}}4|\nabla v_{k}|^2\,d\mathrm{vol}\\
&+\int_{E_{i}}4|\nabla v_{k}|^2(\varphi^2-1)\,d\mathrm{vol}+\int_{E_{i}}Rv_{k}^{2}\,d\mathrm{vol}+\int_{E_{i}}Rv_{k}^{2}(\varphi^2-1)\,d\mathrm{vol}-\int_{E_{i}}\varphi^2v_{k}^2\ln\varphi^{2}\,d\mathrm{vol}\\
&-\int_{E_{i}} v_{k}^2\ln v_{k}^2\,d\mathrm{vol}-\int_{E_{i}} v_{k}^2\ln v_{k}^2(\varphi^{2}-1)\,d\mathrm{vol}.
\end{align*}
Hence, for $k\gg1$ and for some new constant $C$,
\begin{align*}
\mathcal{L}(u_{k}, M, g)-\mathcal{L}(v_{k}, M, g)=&|\mathcal{L}(u_{k}, M, g)-\mathcal{L}(v_{k}, M, g)|\\
\leq&\int_{E_{i}\setminus E_{i-1}}4|\nabla \varphi|^2v_{k}^2\,d\mathrm{vol}+\int_{E_{i}\setminus E_{i-1}}8|\nabla \varphi||\nabla v_{k}|\,d\mathrm{vol}+\int_{E_{i}}4|\nabla v_{k}|^2(\varphi^2-1)\,d\mathrm{vol}\\
&+\int_{E_{i}}\left(\inf_{E_{i}}R\right)_{-}v_{k}^{2}\,d\mathrm{vol}+\int_{E_{i}}v_{k}^2|\varphi^2\ln\varphi^{2}|\,d\mathrm{vol}+\int_{E_{i}} \left|v_{k}^2\ln v_{k}^2\right|\,d\mathrm{vol}\\
\leq&C\int_{E_{i}\setminus E_{i-1}}v_{k}^{2}\,d\mathrm{vol}+C\mathrm{vol}(E_{i}\setminus E_{i-1})\sqrt{\int_{D(o, K_{i})}v_{k}^2\,d\mathrm{vol}}\\
&+\left(\inf_{E_{i}}R\right)_{-}\int_{E_{i}}v_{k}^2\,d\mathrm{vol}+e^{-1}\int_{E_{i}}v_{k}^2\,d\mathrm{vol}+\int_{E_{i}} \left|v_{k}^2\ln v_{k}^2\right|\,d\mathrm{vol},
\end{align*}
where in the last inequality we used also \eqref{Lemma3.1_2} and we are using the standard notation  $(\cdot)_{-}$ for the negative part of a function.

For $k\gg 1$, using Step 1 we can make also all the terms on the RHS less then $\frac{\varepsilon}{10}$. Hence we have obtained that there exists $\overline{k}$ such that $\forall\,k>\overline{k}$
\[
\ \mathcal{L}(u_k,M, g)-\mathcal{L}(v_{k},M, g)\leq\frac{\varepsilon}{2}
\]
In particular, for all $k>\overline{k}$ we get that
\begin{equation*}
\lambda(M\setminus E_{i-1})\leq\mathcal{L}(u_{k}, M, g)\leq\mathcal{L}(v_{k}, M, g)+\frac{\varepsilon}{2}=\lambda_{k}+\frac{\varepsilon}{2}.
\end{equation*}
Since $\lambda_{k}\searrow\lambda$ as $k\to\infty$ we also know that there exists a $\overline{\overline{k}}$ such that $\forall\,k>\overline{\overline{k}}$ we have that
\[
\ \lambda_{k}<\lambda+\frac{\varepsilon}{2}.
\]
Thus, for $k>\tilde{k}(i)=\max\left\{\overline{k}, \overline{\overline{k}}\right\}$,  we get that $\mathcal{L}(u_{k}, M, g)<\lambda+\varepsilon$, concluding the proof of Step 2. 

In particular, we deduce that
\[
\ \lambda(M\setminus E_{i-1})<\lambda+\varepsilon.
\]
\medskip

From this it follows that
\[
\ \lambda_{\infty}=\liminf_{i\to\infty}\lambda(M\setminus E_{i})\leq \lambda +\varepsilon,
\]
and, since by the assumption $\lambda<\lambda_{\infty}$ we can choose $0<\varepsilon<\lambda_{\infty}-\lambda$, this leads to the desired contradiction and hence to the validity of the lemma. 
\end{proof}
\medskip

In particular we have obtained that there exists a $J$ such that $\forall\,j\geq J$
\begin{equation}\label{nonzero}
\int_{E_{j}}v^{2}\,d\mathrm{vol}\geq\delta>0.
\end{equation}

By Lemma \ref{2.1} applied to the $v_{k}$ we know also that $\left\{v_{k}\right\}_{k>K_{j}}$ is bounded in $C^{1}(E_{j})$. Hence, up to a subsequence $v_{k}\to v$ on $E_{j}$ in $C^{0,\alpha}$, for $0<\alpha<1$, and $v\in C^{0,\alpha}(E_{j})$.

By \eqref{nonzero}, we know that there exists $x\in E_{j}$ such that $v(x)>0$. Hence $0\leq v\in C^{0,\alpha}_{\mathrm{loc}}(M)$ and there exists $x\in M$ such that $v(x)>0$.

Note that $v$ is bounded in $M$. In fact, for any fixed $x\in M$, we have that $x\in E_{j}$ for $j\gg1$. Consider $\left\{v_{k}\right\}_{k>K_{j}}$ on $E_{j}$. Since by Lemma \ref{2.1} we have that $v_{k}\leq C$, for every $\varepsilon>0$
\begin{equation}\label{bdd_v}
|v(x)|\leq |v(x)-v_{k}(x)|+ |v_{k}(x)|\leq C+\varepsilon,
\end{equation}
for $k$ large enough.

\begin{lemma}\label{WeakSol}
$v$ is a weak solution of \eqref{EL}
\end{lemma}

\begin{proof}
Fix $\varphi\in C_{c}^{\infty}(M)$ and let $k_0$ be such that $\operatorname{supp} \varphi\subset D(o,k_0)$. Using that, for every $k>k_0$, 
\[
\ 4\int \left\langle \nabla \varphi, \nabla v_{k} \right\rangle \,d\mathrm{vol}+\int \RRR v_{k}\varphi \,d\mathrm{vol}
-\int v_{k}\varphi\ln v_{k}^{2}\,d\mathrm{vol}=\int\lambda_{k}v_{k}\varphi \,d\mathrm{vol},
\]
we want to prove that
\[
\ 4\int \left\langle \nabla \varphi, \nabla v \right\rangle\,d\mathrm{vol}+\int \RRR v\varphi\,d\mathrm{vol}
-\int v\varphi\ln v_{k}^{2}\,d\mathrm{vol}=\int\lambda\varphi\,d\mathrm{vol}.
\]
By the fact that $v_{k}\to v$ weakly in $W^{1,2}(M)$, one gets that
\begin{align*}
\int\left\langle \nabla\varphi, \nabla v_{k}\right\rangle\,d\mathrm{vol}\to&\int\left\langle \nabla \varphi, \nabla v\right\rangle\,d\mathrm{vol}\\
\int Rv_{k}\varphi\,d\mathrm{vol}\to&\int Rv\varphi\,d\mathrm{vol},
\end{align*}
as $k\to\infty$. Moreover
\[
\ \int\lambda_{k}v_{k}\varphi\,d\mathrm{vol}-\int\lambda v\varphi\,d\mathrm{vol}=\int(\lambda_{k}-\lambda)v_{k}\varphi\,d\mathrm{vol}-\int\lambda(v_{k}-v)\varphi\,d\mathrm{vol},
\]
and
\[
\ \int(\lambda_{k}-\lambda)v_{k}\varphi\,d\mathrm{vol}\leq\left|\lambda_{k}-\lambda\right|\left|\int v_{k}\varphi\,d\mathrm{vol}\right|\leq 2\left|\int v\varphi\,d\mathrm{vol}\right|\left|\lambda_{k}-\lambda\right|\to 0,
\]
as $k\to \infty$.  Hence, it remains only to prove that $\forall\,\delta>0$, there exists  $\tilde{k}>k_0$ such that 
\begin{equation}\label{ABCD}
\left|\int v_{k}\ln v_{k}\varphi\,d\mathrm{vol}
-\int v\ln v \varphi\,d\mathrm{vol}\right| \leq\delta
\end{equation}
 for every $k>\tilde{k}$. For every $\varepsilon>0$ we write
\begin{align*}
\left|\int v_{k}\ln v_{k}\varphi\,d\mathrm{vol}
-\int v\ln v \varphi\,d\mathrm{vol}\right| =& \left| \int v_{k}\ln v_{k}\varphi\,d\mathrm{vol}-\int v_{k}\ln (v_{k}+\varepsilon)\varphi\,d\mathrm{vol}\right|\\
&+ \left| \int v_{k}\ln (v_{k}+\varepsilon)\varphi\,d\mathrm{vol}-\int v_{k}\ln (v+\varepsilon)\varphi\,d\mathrm{vol}\right|\\
&+\left| \int v_{k}\ln (v+\varepsilon)\varphi\,d\mathrm{vol}-\int v\ln (v+\varepsilon)\varphi\,d\mathrm{vol}\right|\\
&+\left| \int v\ln (v+\varepsilon)\varphi\,d\mathrm{vol}-\int v\ln v\varphi\,d\mathrm{vol}\right|.
\end{align*}
Let us first observe that
\begin{align*}
\left| \int v_{k}\ln v_{k}\varphi\,d\mathrm{vol}-\int v_{k}\ln (v_{k}+\varepsilon)\varphi\,d\mathrm{vol}\right|=&\left|\int v_{k}\ln\left(\frac{v_{k}+\varepsilon}{v_{k}}\right)\varphi\,d\mathrm{vol}\right|\\
\leq&\int v_{k}|\varphi|\left|\ln\left(\frac{v_{k}+\varepsilon}{v_{k}}\right)\right|\,d\mathrm{vol}\\
=&\int v_{k}|\varphi|\ln\left(1+\frac{\varepsilon}{v_{k}}\right)\,d\mathrm{vol}\leq \int\varepsilon|\varphi|\,d\mathrm{vol}.
\end{align*}
Hence there exists an $\varepsilon_{0}=\varepsilon_{0}(\varphi)$ such that $\forall\,\varepsilon<\varepsilon_{0}$ and for all $k\in \mathbb N$,
\begin{equation}\label{Aweak}
\left| \int v_{k}\ln v_{k}\varphi\,d\mathrm{vol}-\int v_{k}\ln (v_{k}+\varepsilon)\varphi\,d\mathrm{vol}\right|<\frac{\delta}{4}.
\end{equation}
We claim now that there exists a constant $C=C(\varepsilon_{0})$ such that $v\ln(v+\varepsilon)<C$ for every $\varepsilon<\varepsilon_{0}$. Indeed, where $v>1-\varepsilon_{0}$, we have that
\begin{align*}
\left|\ln(v+\varepsilon)\right|\leq&\max\left\{\ln(v+\varepsilon), \left|\ln(1-\varepsilon_{0})\right|\right\}\\
\leq&\max\left\{\ln(v+\varepsilon_{0}), \left|\ln(1-\varepsilon_{0})\right|\right\}.
\end{align*}
Hence, using also \eqref{bdd_v},
\begin{align*}
\left|v\ln(v+\varepsilon)\right|=& |v||\ln(v+\varepsilon)|\leq(v+\varepsilon_{0})|\ln(v+\varepsilon)|\\
\leq&\max\left\{(v+\varepsilon_{0})\ln(v+\varepsilon_{0}), (v+\varepsilon_{0})|\ln(1-\varepsilon_{0})|\right\}\\
\leq&\max\left\{(C+\varepsilon_{0})\ln(C+\varepsilon_{0}),\left|\ln(1-\varepsilon_{0})\right|(C+\varepsilon_{0})\right\}.
\end{align*}
On the other hand, where $v\leq 1-\varepsilon_{0}$, we have that for every $\varepsilon<\varepsilon_{0}$, $v+\varepsilon<v+\varepsilon_{0}\leq 1$. Hence
\[
\ \left|v\ln(v+\varepsilon)\right|=\left|v\right|\left|\ln(v+\varepsilon)\right|<|v||\ln v| \leq e^{-1},
\]
and there exists a $C=C(\varepsilon_{0})$ such that for every $\varepsilon<\varepsilon_{0}$ 
\[
\ |v\ln(v+\varepsilon)|<C,
\]
proving the claim.

Hence, $|v\ln(v+\varepsilon)\varphi|\leq C|\varphi|\in L^{1}$ and by the dominated convergence theorem we get 
\begin{equation*}
\int v \ln(v+\varepsilon)\varphi\,d\mathrm{vol}\to\int v\ln v\varphi\,d\mathrm{vol}
\end{equation*}
as $\varepsilon\to 0$ and, in particular, there exists an $\varepsilon_{1}\leq\varepsilon_{0}$ such that $\forall\,\varepsilon\leq\varepsilon_{1}$
\begin{equation}\label{Dweak}
\left| \int v\ln (v+\varepsilon)\varphi\,d\mathrm{vol}-\int v\ln v\varphi\,d\mathrm{vol}\right|\leq\frac{\delta}{4}.
\end{equation}
We now compute that
\begin{equation*}
\left| \int v_{k}\ln (v+\varepsilon_{1})\varphi\,d\mathrm{vol}-\int v\ln (v+\varepsilon_{1})\varphi\,d\mathrm{vol}\right|\leq \left\|\varphi\right\|_{2}\left\|v_{k}-v\right\|_{2}\sup \left|\ln(v+\varepsilon_{1})\right|
\end{equation*}
Since by Lemma \ref{2.1}
\[
\ \left|\ln(v+\varepsilon_{1})\right|\leq\max\left\{\left|\ln \varepsilon_{1}|, |\ln(C+\varepsilon_{1})\right|\right\},
\]
we hence get that for $k\gg1$
\begin{equation}\label{Cweak}
\left| \int v_{k}\ln (v+\varepsilon_{1})\varphi\,d\mathrm{vol}-\int v\ln (v+\varepsilon_{1})\varphi\,d\mathrm{vol}\right|\leq \frac{\delta}{4}.
\end{equation}
Finally, by Lemma \ref{2.1}, we have that
\begin{align}
\left| \int v_{k}\ln (v_{k}+\varepsilon_{1})\varphi\,d\mathrm{vol}-\int v_{k}\ln (v+\varepsilon_{1})\varphi\,d\mathrm{vol}\right|=&\int\left|v_{k}\ln\left(1+\frac{v_{k}-v}{v+\varepsilon_{1}}\right)\varphi\,d\mathrm{vol}\right|\label{Bweak}\\
\leq&\int|v_{k}|\left|\frac{v_{k}-v}{v+\varepsilon_1}\right||\varphi|\,d\mathrm{vol}\nonumber\\
\leq& \frac{C}{\varepsilon_1}\left\|v_{k}-v\right\|_{2}\left\|\varphi\right\|_{2}\leq\frac{\delta}{4} \nonumber
\end{align}
for $k$ large enough. By \eqref{Aweak}, \eqref{Bweak}, \eqref{Cweak}, \eqref{Dweak} we obtain \eqref{ABCD} and hence our claim.
\end{proof}
\medskip

By Lemma \ref{WeakSol} and elliptic regularity (see e.g. Theorem 3.54 in \cite{Au}), we get that $v\in C^{2,\beta}_{\mathrm{loc}}$, for $\beta<\alpha$. Thus an application of the maximum principle gives that $v>0$ on $M$ and this permits to get, again by elliptic regularity, that $v$ is indeed smooth and thus a classical solution of \eqref{EL}.
\medskip

To conclude the proof of Theorem \ref{ExtrLSI} it remains to prove the following
\begin{lemma}\label{lem4.3}
$v$ is an extremal of the Log Sobolev functional $\mathcal{L}$.
\end{lemma}

\begin{proof}
Since $v$ is a classical solution of \eqref{EL} we have that
\[
\ 4v\Delta v=\RRR v^2-v^2\ln v^2-\lambda v^2.
\]
Since $v\in W^{1,2}(M)$, $R$ is bounded from below, and \eqref{Rot1} holds, we know that $\min\{v\Delta v;0\}\in L^{1}(M)$. Moreover, by Cauchy-Schwarz inequality we have that $v\nabla v\in L^{1}(M)$. Hence, by Gaffney's version of Stokes' theorem, \cite{Ga}, $v\Delta v\in L^1(M)$ and
\[
\ \int_{M} v\Delta v\,d\mathrm{vol}=-\int_{M} |\nabla v|^2\,d\mathrm{vol}.
\]
Thus
\[
\ \mathcal{L}(v, M, g)=\int_{M} (4|\nabla v|^2+ \RRR v^2- v^{2}\ln v^2 )\,d\mathrm{vol}=\lambda \int_{M} v^2\,d\mathrm{vol}.
\]
If $\int_{M}v^{2}=1$ then $v$ is an extremal function of $\mathcal{L}$. So suppose $\int_{M}v^{2}<1$ and consider the function 
\[
\ \tilde{v}=\frac{v}{\left\|v\right\|_{2}}.
\]
Then $\left\|\tilde{v}\right\|_{2}=1$ and
\begin{align*}
\lambda=&\frac{1}{\left\|v\right\|_{2}^2}\mathcal{L}(v, M, g)=\frac{\int_{M} (4|\nabla v|^2+\RRR v^2-v^2\ln v^2)\,d\mathrm{vol}}{\left\|v\right\|_2^2}\\
=&\int_{M}\left[\left(4|\nabla \tilde{v}|^2+\RRR |\tilde{v}|^2-\tilde{v}^2\ln\tilde{v}^2\right)-\tilde{v}^2\ln\left\|v\right\|_{2}^2\right]\,d\mathrm{vol}.
\end{align*}
We claim that $\mathcal{L}(\tilde{v}, M, g)\geq \lambda$. From this it would follow that $\lambda\geq \lambda - \tilde{v}^2\ln\left\|v\right\|_{2}^{2}$. Since we are assuming $\int_{M}v^2<1$ we hence get a contradiction, thus finishing the proof of Theorem \ref{ExtrLSI}. Let us now prove the claim.
Let $o\in M$ and consider the standard smooth cut-off functions, for $R>1$,
\begin{equation*}
\begin{aligned}
0\leq\varphi_R\leq 1,\quad&\varphi_R|_{B_R(o)}=1,\quad&\varphi_{R}|_{M\setminus B_{2R}(o)}=0,\quad&|\nabla\varphi_{R}|\leq 2
\end{aligned}
\end{equation*}
We then consider the functions $\tilde{v}\varphi_{R}\in C_{c}^{\infty}(M)$ and we prove that
\begin{equation}\label{Extr1}
\mathcal{L}(\tilde{v}\varphi_{R}, M, g)\to \mathcal{L}(\tilde{v},M,g),
\end{equation}
as $R\to\infty$. Since $\mathcal{L}(\tilde{v}\varphi_{R}, M, g)\geq\lambda$ for every $R$, we hence get the claim. 

We have that
\begin{align*}
\mathcal{L}(\tilde{v},M,g)-\mathcal{L}(\tilde{v}\varphi_{R}, M, g)=&\int_{M\setminus B_{R}}(1-\varphi_{R}^2)\left(4|\nabla \tilde{v}|^2+\RRR \tilde{v}^2-\tilde{v}^2\ln\tilde{v}^2\right)\,d\mathrm{vol}\\
&-\int_{B_{2R}\setminus B_R}4\tilde{v}^2|\nabla \varphi_{R}|^2\,d\mathrm{vol}+\int_{B_{2R}\setminus B_R}\tilde{v}^2\varphi_{R}^2\ln\varphi_{R}^2\,d\mathrm{vol}\\
&-8\int_{B_{2R}\setminus B_{R}}\tilde{v}\varphi_R\left\langle\nabla\tilde{v},\nabla\varphi_{R}\right\rangle\,d\mathrm{vol}.
\end{align*}
Using that $\tilde{v}\in W^{1,2}(M)$, $|\nabla \varphi_{R}|$ is bounded and $\varphi_{R}^2\ln\varphi_{R}^2<e^{-2}$, we easily get that the last three terms tend to zero as $R\to \infty$. Moreover, since
\[
\ \tilde{v}^2\ln\tilde{v}^2=-4\tilde{v}\Delta \tilde{v}+\RRR\tilde{v}^2-\lambda \tilde{v}^2-\tilde{v}^2\ln\left\|v\right\|_{2}^2,
\]
we have that
\begin{equation*}
\int_{M\setminus B_{R}}(1-\varphi_{R}^2)\left(4|\nabla \tilde{v}|^2+\RRR \tilde{v}^2-\tilde{v}^2\ln\tilde{v}^2\right)\,d\mathrm{vol}=\int_{M\setminus B_{R}}(1-\varphi_{R}^2)\left(4|\nabla \tilde{v}|^2+4\tilde{v}\Delta\tilde{v}+\lambda\tilde{v}^2+\tilde{v}^2\ln\left\|v\right\|_{2}^{2}\right)\,d\mathrm{vol}
\end{equation*}
Hence, using again the fact that $\tilde{v}\in W^{1,2}(M)$, $\tilde{v}\in L^{2}(M)$ and $\tilde{v}\Delta\tilde{v}\in L^{1}(M)$, it is not difficult to prove that also this expression goes to $0$ as the radius $R\to\infty$,
thus proving the claim.
 \end{proof}

\section{Distance-like functions on complete non-compact manifolds}\label{AppA}

Before getting into the proof of Proposition \ref{DistLike} we would like to explicitly point out that, using the distance-like functions we are going to construct, one is able to produce Hessian cut-off functions. In particular we can obtain the following result extending Proposition 3.7 in \cite{GP}; see also \cite{G}.

\begin{corollary}[Hessian cut-off functions]\label{coro-hess}
Let $(M^{m}, g)$ be a complete non-compact Riemannian manifold such that $|\mathrm{Ric}|\leq (m-1) K$ for some $K\in\left(0,\infty\right)$ and $\mathrm{inj}_{(M,g)}>i_{0}>0$ and fix a reference point $o\in M$. Then there exist a constant $C_{m,K}\in\left(1,\infty\right)$, depending only on $m$ and $K$, and a sequence $\{\chi_n\}\subset C_{c}^{\infty}(M,[0,1])$ of cut-off functions such that
\begin{itemize}
\item $\chi_n=1$ on $B_{n-C_{m,K}}(o)$
\item $\operatorname{supp}(\chi_n)\subset B_{2n-1}(o)$. 
\item $\left\|\nabla \chi_n\right\|_{\infty}\to 0$ as $n\to\infty$
\item $\left\|\mathrm{Hess}(\chi_n)\right\|_{\infty}\to 0$ as $n\to\infty$.
\end{itemize}
\end{corollary}

\begin{proof}(of Corollary \ref{coro-hess})
Let $\phi\in C^\infty (\mathbb{R},[0,1])$ be a cut-off function such that 
\begin{equation*}
\begin{aligned}
\phi|_{(-\infty,0]}=1,\quad&\phi|_{[1,\infty)}=0,&|\phi'|\leq 2,\quad|\phi''|\leq a,
\end{aligned}
\end{equation*}
for some $a>0$. For any $n\geq 1$, let $\phi_n\in C^\infty ([0,+\infty))$ be a cut-off function defined by $\phi_n(t):=\phi(t/n-1)$. In particular
\begin{equation*}
\begin{aligned}
\phi_n|_{[0,n]}=1,\quad&\phi_n|_{[2n,\infty)}=0,\quad&|\phi'_n|\leq 2/n,\quad&|\phi''_n|\leq a/n^2.
\end{aligned}
\end{equation*}
 Let the constant $C_{m,K}$ and the function $h\in C^\infty(M)$ be as in Proposition \ref{DistLike}. For each integer $n> C_{m,K}$, define $\chi_n:=\phi_n\circ h$. One can easily see that 
$\chi_n=1$ on $B_{n-C_{m,K}}(o)$ and $\chi_n=0$ on $M\setminus B_{2n-1}(o)$. Moreover
\begin{eqnarray*}
|\nabla \chi_n|&=&|\phi'(h)||\nabla h|\leq 2C_{m,K}/n\\
|\mathrm{Hess}(\chi_n)|&\leq& |\phi_n''(h)(dh\otimes dh)|+|\phi'_n(h)\mathrm{Hess}(h)|\leq aC_{m,K}^2/n^2+2C_{m,K}/n.
\end{eqnarray*}
\end{proof}
\medskip

The remaining of the section is devoted to present the proof of Proposition \ref{DistLike}.

The idea of the proof, coming from a paper by Tam \cite{T} (see e.g. \cite{CCG_3}), is to start with a distance-like function with uniformly bounded gradient and evolve it by heat equation to have also a uniform bound on the Hessian. The key point which permits to conclude in our situation is the employment of a result of M. Anderson controlling the global $C^{1,\frac{1}{2}}$ harmonic radius of the manifold under our assumptions.
\medskip 

By Proposition 2.1 in \cite{GW} we know that given a complete Riemannian manifold $(M^{m}, g)$ and $o\in M$ there exists a $u\in C^{\infty}(M)$ with
\begin{equation}\label{GreeneWu}
|u(x)-r(x)|\leq 1\qquad\mathrm{and}\qquad|\nabla u|(x)\leq 2\quad\mathrm{on}\quad M.
\end{equation}
Let $H:M\times M\times(0,\infty)\to(0,\infty)$ be the heat kernel of $(M,g)$. Then the function $h:M\times(0,\infty)\to\mathbb{R}$ defined by
\begin{equation}\label{conv}
h(x,t):=\int_{M}H(x,y,t)u(y)d\mathrm{vol}(y)
\end{equation}
is a solution to the heat equation with $\lim_{t\to 0}h(x, t)=u(x)$ uniformly in $x\in M$.
\medskip

Reasoning as in Step 1 and Step 2 of the proof of Proposition 26.49 in \cite{CCG_3} we obtain that $h(x):= h(x,1)$ satisfies both conditions \eqref{A} and \eqref{B} of the statement. Indeed sectional curvature bounds can be replaced in these steps by Ricci curvature bounds, where needed, without any further modifications. To obtain the control on the Hessian we can adapt Step 3 of Proposition 26.49, up to use a result obtained in \cite{And} instead of a result by J. Jost and H. Karcher, \cite{JK}. For the sake of completeness we give a detailed exposition of this last step below.

It hence remains to prove that there exists a constant $C=C(m,k)$ such that
\[
\ \left|\mathrm{Hess}(h)\right|(x, 1)\leq C\quad\forall x\in M.
\]
Given any point $x\in M$, we know that
\[
\ \mathrm{exp}_{x}:\mathbb{B}_{i_{0}}(0)\subset T_{x}M\to B_{i_{0}}(x)
\]
is a local diffeomorphism, where $i_{0}$ is the lower bound for the injectivity radius. We consider the pull-back metric
\[
\ \hat{g}:=(\mathrm{exp}_{x})^{*}g
\] 
on $\mathbb{B}_{i_{0}}(0)$.
Define $\hat{h}_{x}:\mathbb{B}_{i_{0}}(0)\times(0, \infty)\to\mathbb{R}$ by
\begin{equation*}
\hat{h}_{x}(v, t):=h\left(\mathrm{exp}_{x}(v), t\right)-u(x).
\end{equation*}
Then, by \eqref{A} and \eqref{GreeneWu}, there exists $C_{1}$ depending only on $m$ and $K$ such that, for every $x\in M$,
\[
\ \left|\hat{h}_{x}(0,t)\right|=\left|h(x,t)-u(x)\right|\leq C_{1}
\]
for $t\in (0, 1]$. Moreover, by \eqref{B}, for every $x\in M$,
\[
\ \left|\hat{h}_{x}(v, t)\right|\leq C_{1}+i_{0}C_{m,K},
\]
for every $t\in(0,1]$ and $v\in T_{x}M$ with $|v|<i_{0}$. By the invariance by isometries of the norm of the Hessian, note that
\[
\ \left|\mathrm{Hess}_{g}h \right|_{g}(x, t)=\left|\mathrm{Hess}_{\hat{g}}\hat{h}_{x}\right|_{\hat{g}}(0, t).
\]
Now we have that
\[
\ \left(\partial_{t}-\Delta_{\hat{g}}\right)\hat{h}_{x}=0,
\]
and
\begin{equation}\label{Est}
|\hat{h}_{x}|\leq C_{1}+i_{0}C_{m,K}\qquad\mathrm{in}\qquad\mathbb{B}_{i_{0}}(0)\times(0,1].
\end{equation}
Recall that a local coordinate system $\left\{x^{i}\right\}$ is said to be harmonic if for any $i$, $\Delta_{g}x^{i}=0$. The harmonic radius is then defined as follows.
\begin{definition}
Let $(M^m, g)$ be a smooth Riemannian manifold and let $x\in M$. Given $Q>1$, $k\in\mathbb{N}$, and $\alpha\in\left(0,1\right)$, we define the $C^{k,\alpha}$ \textbf{harmonic radius at $x$} as the largest number $r_{H}=r_{H}(Q,k,\alpha)(x)$ such that on the geodesic ball $B_{r_H}(x)$ of center $x$ and radius $r_{H}$, there is a harmonic coordinate chart such that the metric tensor is $C^{k,\alpha}$ controlled in these coordinates. Namely, if $g_{ij}$, $i,j=1,\ldots,m$, are the components of $g$ in these coordinates, then
\begin{enumerate}
\item $Q^{-1}\delta_{ij}\leq g_{ij}\leq Q\delta_{ij}$ as bilinear forms;
\item $\sum_{1\leq|\beta|\leq k}r_{H}^{|\beta|}\sup\left|\partial_{\beta}g_{ij}(y)\right|+\sum_{|\beta|=k}r_{H}^{k+\alpha}\sup_{y\neq z}\frac{\left|\partial_{\beta}g_{ij}(z)-\partial_{\beta}g_{ij}(y)\right|}{d_{g}(y,z)^{\alpha}}\leq Q-1$.
\end{enumerate}
We then define the \textbf{(global) harmonic radius} $r_{H}(Q,k,\alpha)(M)$ of $(M,g)$ by
\[
\ r_{H}(Q,k,\alpha)(M)=\inf_{x\in M}r_{H}(Q,k,\alpha)(x)
\]
where $r_{H}(Q,k,\alpha)(x)$ is as above.
\end{definition}
The following result has been proved in \cite{And}. As it is stated below, it can be found in the survey paper \cite{HH} (see also \cite{Heb}).
\begin{proposition}\label{HarmRadEst}
Let $\alpha\in(0,1)$, $Q>1$, $\delta>0$. Let $(M^m,g)$ be a smooth Riemannian manifold, and $\Omega$ an open subset of $M$. Set
\[
\ \Omega(\delta)=\left\{x\in M\quad\mathrm{s.t.}\quad d_{g}(x,\Omega)<\delta\right\}.
\]
Suppose that for some $\lambda\in\mathbb{R}$ and some $i>0$, we have that for all $x\in\Omega(\delta)$,
\[
\ \mathrm{Ric}(x)\geq\lambda g(x)\quad\mathrm{and}\quad\mathrm{inj}_{g}(x)\geq i.
\]
Then there exists a positive constant $C=C(m,Q,\alpha,\delta,i,\lambda)$, such that for any $x\in \Omega$, $r_{H}(Q, 0, \alpha)(x)\geq C$. In addition, if we furthermore assume that for some integer $k$ and some positive constant $C(j)$,
\[
\ |\nabla^{j} \mathrm{Ric}(x)|\leq C(j)\quad\mathrm{for\,\, all}\quad j=0,\ldots,k \quad\mathrm{and \,\,all}\quad x\in\Omega(\delta),
\]
then, there exists a positive constant $C=C(m,Q,k,\alpha,\delta, i, C(j)_{0\leq j\leq k})$, such that for any $x\in\Omega$, $r_{H}(Q, k+1,\alpha)(x)\geq C$.
\end{proposition}

In our situation, since $\mathrm{inj}_{\hat{g}}(0)\geq i_{0}$, we have that for every $v\in\mathbb{B}_{\frac{i_{0}}{2}}(0)$ 
\[
\ \mathrm{inj}_{\hat{g}}(v)=d_{\hat{g}}\left(v, \partial\mathbb{B}_{i_{0}}(0)\right)\geq \frac{i_{0}}{2}.
\]
Hence, for every $v\in\mathbb{B}_{\frac{i_{0}}{2}}(0)$ we have that
\begin{eqnarray*}
	&&|\mathrm{Ric}(\hat{g})|(v)\leq (m-1)K,\\
	&&\mathrm{inj}_{\hat{g}}(v)\geq\frac{i_{0}}{2}.
\end{eqnarray*}
Fixing $\delta>0$ sufficiently small, by Proposition \ref{HarmRadEst} there exists $C=C(m,Q,\delta,i_{0}, K)$ such that for every $w\in\mathbb{B}_{\frac{i_{0}}{2}-\delta}(0)=: \mathbb{B}_{\rho_{0}}(0)$
\[
\ r_{H}(Q, 1, \frac{1}{2})(w)\geq C.
\]
With respect to the harmonic coordinates $\left\{\hat{y}_{i}\right\}$,
\begin{align*}
\frac{\partial \hat{h}_{x}}{\partial t}=&\frac{1}{\sqrt{|\hat{g}|}}\sum_{i,j=1}^{m}\frac{\partial}{\partial \hat{y}^{i}}\left(\sqrt{|\hat{g}|}\hat{g}^{ij}\frac{\partial\hat{h}_{x}}{\partial\hat{y}^{j}}\right)\\
=&\sum_{i,j=1}^{m}\left(\hat{g}^{ij}\frac{\partial^{2}\hat{h}_{x}}{\partial \hat{y}^{i}\partial\hat{y}^{j}}+\frac{1}{\sqrt{|\hat{g}|}}\frac{\partial}{\partial\hat{y}^{i}}\left(\sqrt{|\hat{g}|}\hat{g}^{ij}\right)\frac{\partial\hat{h}_{x}}{\partial \hat{y}^{j}}\right)\\
=&\sum_{i,j=1}^{m}\left(\hat{g}^{ij}\frac{\partial^{2}\hat{h}_{x}}{\partial\hat{y}^{i}\partial\hat{y}^{j}}\right).
\end{align*}
Using \eqref{Est} and the definition of harmonic radius, by Schauder's estimates for parabolic equations (see \cite{Fried}), we get that
\[
\ \left\|\hat{h}_{x}\right\|_{C^{2,\frac{1}{2}}}\leq C_{2}<\infty
\]
in $\mathbb{B}_{\frac{\rho_{0}}{2}}(0)\times\left[\frac{1}{2}, 1\right]$, for some $C_{2}$ depending only on $m$ and $K$. In particular
\[
\ \left\|\hat{h}_{x}(\cdot, 1)\right\|_{C^{2, \frac{1}{2}}}\leq C_{2}
\]
in $\mathbb{B}_{\frac{\rho_0}{2}}(0)$ with respect to harmonic coordinates. Since
\begin{eqnarray*}
\,^{\hat{g}}\nabla_{i}\,^{\hat{g}}\nabla_{j}&=&\frac{\partial^{2}}{\partial \hat{y}^{i}\partial\hat{y}^{j}}-\,^{\hat{g}}\Gamma_{ij}^{k}\frac{\partial}{\partial y^{k}}\\
\,^{\hat{g}}\Gamma_{ij}^{k}&=&\frac{1}{2}\hat{g}^{kl}\left(\frac{\partial}{\partial \hat{y}^{i}}\hat{g}_{jl}+\frac{\partial}{\partial\hat{y}^{j}}\hat{g}_{il}-\frac{\partial}{\partial \hat{y}^{l}}\hat{g}_{ij}\right),
\end{eqnarray*}
corresponding to $\hat{y}=0$, by the definition of harmonic radius, we have
\[
\ \left|\mathrm{Hess}_{g}h\right|_{g}(x, 1)=\left|\mathrm{Hess}_{\hat{g}}(\hat{h}_x)\right|_{\hat{g}}(0,1)\leq C_{3}
\]
for every $x\in M$, with $C_{3}$ depending only on $m$ and $K$. This concludes the proof of the proposition.

\section{Proof of Theorem \ref{ExtrLSI}: second part}\label{ThExtrII}

By Proposition \ref{DistLike}, we can readily extend Lemma 2.3 in \cite{ZHQ} to the following situation.
\begin{lemma}\label{2.3}
Let $(M^{m}, g)$ be a complete Riemannian  manifold such that
\begin{equation*}
|\mathrm{Ric}|\leq (m-1)K\qquad\mathrm{and}\qquad\mathrm{inj}_{(M,g)}\geq i_{0}>0.
\end{equation*}
Let $u$ be a bounded subsolution to \eqref{EL} on $M$ such that $\left\|u\right\|_{L^2(M)}\leq 1$. Let $o$ be a reference point on $M$. Then there exist positive numbers $r_{0}, a$ and $A$, which may depend on $K, i_{0}$ and the location of the reference point such that
\begin{equation*}
\ u(x)\leq Ae^{-ad^{2}(x, o)},
\end{equation*}
when $d(x, o)\geq r_{0}$.
\end{lemma}
Since the extremal function $v$ produced in the proof of the first part of Theorem \ref{ExtrLSI} is in particular a bounded solution of \eqref{EL} with $\left\|v\right\|_{L^2(M)}\leq 1$ we get the desired exponential decay \eqref{ExpDec} and thus the second part of Theorem \ref{ExtrLSI}.

\section{Gradient Ricci soliton structure on Ricci solitons}\label{GRSRS}

\subsection{Growth of the soliton field}

In this subsection we prove a general upper bound for the growth of the soliton field $X$ of a generic (not necessarily gradient nor shrinking) Ricci soliton $(M,g)$.

Let us first recall that, by a standard computation, the soliton equation gives a control on the tangential part of $X$ along geodesics. Namely we have the following
\begin{lemma}\label{lem_radial}
In the assumptions of Theorem \ref{th_SolGrowth},
for any unit-speed geodesic $\gamma:[0,L]\to M$ it holds
\[
\left|\frac{d}{dt}\,g( \dot\gamma (t), X(\gamma(t)))\right| \leq |\lambda_S| + (m-1)K.
\]
In particular 
\begin{equation*}
\left|g( \dot\gamma (L), X(\gamma(L)))- g( \dot\gamma (0), X(\gamma(0)))\right|\leq L(|\lambda_S| + (m-1)K).
\end{equation*} 
\end{lemma}

\begin{proof}
We compute
\begin{align*}
\frac{d}{dt}\,g( \dot\gamma (t), X(\gamma(t))) =& g( \nabla_{\dot\gamma}\dot\gamma (t), X(\gamma(t))) +g( \dot\gamma (t), \nabla_{\dot\gamma}X(\gamma(t)))\\
=&\frac{1}2 \mathcal L_X g\,(\dot\gamma(t),\dot\gamma(t))\\
=&(\lambda_S g - \Ric)(\dot\gamma(t),\dot\gamma(t)).
\end{align*}
\end{proof}

The idea of the proof of Theorem \ref{th_SolGrowth} goes as follows. By continuity $|X|$ is bounded on the unitary geodesic ball $B_{1}(o)\subset M$.
One can hence apply Lemma \ref{lem_radial} along all the   geodesics $\gamma_y$ connecting $y \in B_{1}(o)$ to $q\in M$. Suppose that the family of vectors $\{\dot\gamma_y(q)\}_{y \in B_{1}(o)}$ covers an angle in $T_qM$ that is large enough. Then one can obtain a quantitative control of $|X|(q)$ in terms of the $g(\dot\gamma_y , X)$'s.

Accordingly, we need the following estimate of independent interest.

\begin{lemma}[Ricci Hinge Lemma]\label{hinge}
Let $(M^m,g)$ be a complete Riemannian manifold satisfying \linebreak$\Ric\geq-(m- 1)K$ for some constant $K\geq 0$. Let $o,q$ be points of $M$ with $d(o,q)=r>1$. Let $W\in S_qM$, with $S_qM$ denoting the set of unitary vectors in $T_qM$. Then there exists $Z\in S_q M$ such that $\exp_q(sZ)\in B_{1}(o)$ for some $s\in[r-1,r+1]$ and 
\begin{equation}\label{eq_hinge}
|g( Z,W)|\geq\begin{cases} C_0 r^{1-m},&\text{if }K=0,\\
C_1e^{-(m-1)\sqrt K r},&\text{if }K>0,
\end{cases}
\end{equation}
where 
\[
C_0= \frac{1}{2^{m+2}\omega_{m-2}}\vol(B_1(o)),\qquad
C_1=\frac{\left(\sqrt K e^{-\sqrt K}\right)^{m-1}}{8\omega_{m-2}}\vol(B_1(o)),
\]
and $\omega_{m-2}$ is the $(m-2)$-Hausdorff measure of $\mathbb S^{m-2}$. 
\end{lemma}

\begin{remark}
Note that, if the sectional curvature is bounded from below by $-K$, then hinge version of Toponogov's comparison theorem (see for instance \cite[Theorem 79]{Pet}) gives that for all $V,U\in S_qM$, $$d(\exp_q(rV),\exp_q(rU))\leq 1$$ provided that 
\[
|g( V,U)| \geq\begin{cases} 1-\frac{1}{2r^2}&\text{if }K=0,\\
1-\frac{\cosh \sqrt K -1}{\sinh^2(\sqrt Kr)},&\text{if }K>0.
\end{cases}
\]
Accordingly, the estimate \eqref{eq_hinge} can be improved in this case to
\begin{equation*}
|g( Z,W)|\geq\begin{cases} C'_0 r^{-1},&\text{if }K=0,\\
C'_1e^{-\sqrt K r},&\text{if }K>0.
\end{cases}
\end{equation*}
\end{remark}

\begin{remark}
Beyond the proof of Theorem \ref{th_SolGrowth}, Lemma \ref{hinge} can be applied more generally to estimate the growth of any vector field $X$ along which one can control $\mathcal L_Xg$, as for instance a Killing vector field. Namely one can obtain the following estimate
\begin{corollary}\label{coro_kill}
Let $(M,g)$ be a complete non-compact $m$-dimensional Riemannian manifold such that $\Ric \geq -(m-1)K$ for some constant $K\geq 0$ and let $X$ be a smooth Killing vector field on $M$. For any reference point $o\in M$ and for all $q\in M$ it holds
\begin{equation*}
|X|(q)\leq \begin{cases}
X^\ast\, C_0^{-1} d(q,o)^{m-1},&\text{if }K=0,\\
X^\ast\, C_1^{-1} e^{(m-1)\sqrt Kd(q,o)},&\text{if }K>0,
\end{cases}
\end{equation*}
with $C_0$, $C_1$ defined as in Lemma \ref{hinge} and $X^\ast$ defined as in Theorem \ref{th_SolGrowth}.
\end{corollary}

\end{remark}

\begin{proof}[Proof (of Lemma \ref{hinge})]
Set $A_1:=B_1(o)$ and $A_0:=\{q\}$. For $t\in[0,1]$ define 
\[
A_t:=\{x\in M\ :\ \exists\, \gamma:[0,1]\to M\text{ minimal geodesic with }\gamma(0)=q,\ \gamma(1)\in A_1\text{ and }\gamma(t)=x\}.
\]

The lower bound on the Ricci curvature and Brunn-Minkowski's inequality give that, for all $t\in[0,1]$,
\begin{equation}\label{Brunn-Minkowski}
\vol(A_t)\geq t\frac{\sinh^{m-1}(\sqrt{K}tD)}{\sinh^{m-1}(\sqrt{K}D)}\vol(A_1), 
\end{equation}
where $D:=\sup_{y\in A_1}d(q,y)$ satisfies $r\leq D\leq r+1$; see \cite[Theorem 3.2]{Ohta} or \cite[Proposition 1.4.11]{DaiWei}. Let $\e>0$ be small enough, to be chosen later. Setting $t=\e/r$ in \eqref{Brunn-Minkowski} yields
\begin{equation*}
\vol(A_{\e/ r})\geq \frac{\e^m} {r^m}\frac{(\sqrt{K}D)^{m-1}}{\sinh^{m-1}(\sqrt{K}D)}\vol(A_1)\geq \frac{\e^m} {r}\frac{(\sqrt{K})^{m-1}}{\sinh^{m-1}(\sqrt{K}(r+1))}\vol(A_1).
\end{equation*}
For $R>0$, let $\mathbb B_R(0)$ be the ball of radius $R$ centered at $0$ in $\mathbb R^n\cong T_qM$. For $\e<\mathrm{inj}_{g}(q)$, the rescaled exponential map $\chi_\e:=\exp_q (\e\cdot): \mathbb B_{1}(0)\to B_{\varepsilon}(q)$ is a $C^\infty$-diffeomorphism. Moreover the locally Euclidean character of the Riemannian metric implies that $\e^{-m}\chi_\e^\ast d\mathrm{vol} \to d\mathrm{vol}_{Eucl}$ as $\e\to 0$. 
In particular, choosing $\e$ small enough (depending on $q$) and defining $\tilde A_{\e/r}:=\chi_\e^{-1}(A_{\e/r})$, we have that

\begin{equation}\label{BM-normal}
\vol_{Eucl}(\tilde A_{\e/ r})\geq \frac1{2\e^m}\vol(A_{\e/ r}) \geq\frac1{2r}\frac{(\sqrt{K})^{m-1}}{\sinh^{m-1}(\sqrt{K}(r+1))}\vol(A_1)\geq C_2 \frac{1}{r}e^{-(m-1)\sqrt K r},
\end{equation}
for some positive constant $C_2=C_2(m,K,o)=K^{\frac{m-1}2}2^{m-2}e^{-(m-1)\sqrt K}\vol(A_1)$ independent from $q$. Moreover, by definition $A_{\e/r}\subset \overline{B_{\e(1+1/r)}(q)}\setminus B_{\e(1-1/r)}(q)$, which in turn implies that $\tilde A_{\e/r}\subset \overline{\mathbb{B}_{(1+1/r)}(0)}\setminus \mathbb{B}_{(1-1/r)}(0)$.
Let $\mathcal A_s$ be the volume measure on $\partial \mathbb B_s(0)$. By Fubini's theorem, we have that
\begin{eqnarray}\label{fubini}
\vol_{Eucl}(\tilde A_{\e/ r})&=&\int_{\mathbb B_{(1+1/r)}(0)\setminus \mathbb B_{(1-1/r)}(0)}\mathds{1}_{\tilde A_{\e/r}}(x) d\mathrm{vol}_{Eucl}(x) \\
&=& \int_{1-1/r}^{1+1/r}\mathcal A_s(\partial \mathbb B_s(0)\cap \tilde A_{\e/r})ds\leq\frac{2}{r}\max_{s\in [1-1/r,1+1/r]}\mathcal A_s(\partial \mathbb B_s(0)\cap \tilde A_{\e/r})\nonumber.
\end{eqnarray}
Combining \eqref{BM-normal} and \eqref{fubini}, we get that there exists $s_0\in [1-1/r,1+1/r]$ such that 
\[
\mathcal A_{s_0}(\partial \mathbb B_{s_0}(0)\cap \tilde A_{\e/r})\geq  \frac{C_2}{2}e^{-(m-1)\sqrt K r}.
\]
A further rescaling gives that the area of 
\[ 
\hat A:=\frac 1{s_0}\left(\partial \mathbb B_{s_0}(0)\cap \tilde A_{\e/r}\right)\subset \partial \mathbb B_1(0) \cong S_qM
\]
satisfies
\[
\mathcal A_1(\hat A)\geq (1+1/r)^{1-m}\frac{C_2}{2}e^{-(m-1)\sqrt K r}\geq \frac{C_2}{2^m}e^{-(m-1)\sqrt K r}.
\]
By construction, for every $V\in \hat A$ it holds that $\exp_q(sV)\in B_{1}(o)$ for some $s\in[r-1,r+1]$. Set
\[
E_\tau:=\left\{V\in S_qM\ : \ |g( V,W)|\leq \frac{C_2}{2^{m+1}\omega_{m-2}}e^{-(m-1)\sqrt K r}=:\tau \right\}.
\]
It remains only to prove that there exists $Z\in \hat A\setminus E_\tau$. To this end, we compute
\[
\mathcal A_1(E_\tau)=\int_{-\tau}^\tau \omega_{m-2}\sqrt{1-t^2}dt< 2\omega_{m-2}\tau.
\]
In particular 
\[
\mathcal A_1(\hat A)> \mathcal A_1(E_\tau),\]
which means that 
$\hat A\setminus E_\tau$ is nonempty.
\end{proof}

\begin{proof}[Proof (of Theorem \ref{th_SolGrowth})]
We present the proof in case $K>0$. The case $K=0$ can be treated similarly. The result is trivial if $X(q)=0$ or if $q\in B_{1}(o)$. Otherwise, an application of Lemma \ref{hinge} with $W=\frac{X(q)}{|X(q)|}$ gives that there exists $Z\in S_q M$ such that 
\begin{equation}\label{lbX}
|g( Z,X(q))|\geq C_1e^{-(m-1)\sqrt K r}|X|(q)
\end{equation}
and $\exp_q(sZ)\in B_{1}(o)$ for some $s\in[r-1,r+1]$. Applying Lemma \ref{lem_radial} along the unit speed geodesic $\gamma_Z(t):=\exp_q(tZ)$, we get
\begin{align*}
\left|g( Z, X(q))\right|&=
\left|g( \dot\gamma_Z (0), X(\gamma_Z(0)))\right|\\
&\leq \left|g( \dot\gamma_Z (s), X(\gamma_Z(s)))\right| +  s(|\lambda_S| + (m-1)K)\\
& \leq
X^\ast+ (r+1)(|\lambda_S| + (m-1)K) .
\end{align*}
This latter, together with \eqref{lbX}, gives the aimed estimate with
\[
C:=(X^\ast+2|\lambda_S| + 2K(m-1))C_1^{-1},
\]
where $C_1$ is the constant of Lemma \ref{hinge}.
\end{proof}

\subsection{Existence of a gradient Ricci soliton structure}

\begin{proof}[Proof of Theorem \ref{th_RicSol}]
We follow the proof given in \cite{ELNM} for compact Ricci solitons, which adopts an ``elliptic'' approach. The additional difficulty here is due to the  final integration by parts, which has to be justified in the noncompact case.

Let $(M, g, X)$ be a complete shrinking Ricci soliton, i.e. \eqref{RS} holds with $\lambda_{S}>0$.
According to Theorem \ref{ExtrLSI} and our assumptions there exists a smooth strictly positive function $\tilde v$ on the rescaled manifold $(M,2\lambda_Sg)$ satisfying
\[
4\Delta_{2\lambda_Sg} v - R_{2\lambda_Sg}v +2v\ln v + \lambda(M,2\lambda_Sg) v =0,
\]
with $\lambda$ being the infimum of the Log Sobolev functional $\mathcal L(v, M, 2\lambda_Sg)$. By a conformal change, we get that
\begin{equation}\label{EL1}
4\Delta_{g} v - R_{g}v +4\lambda_Sv\ln v + 2\lambda_S\lambda(M,2\lambda_Sg) v =0
\end{equation}
on $(M,g)$. Setting 
$f=-\frac{m}{2}\ln(4\pi)-2\ln v$, we get that $f$ satisfies
\begin{equation}\label{eq-v-RS}
2\Delta f +R - |\nabla f|^2 + 2\lambda_Sf = 2\lambda_S\left[\lambda(M,2\lambda_Sg) - \frac{m}{2}\ln(4\pi)\right].
\end{equation}
We remark in particular that the RHS of this latter inequality is a constant function. Accordingly, recalling that $div_f(\cdot):=e^f div(e^{-f}\cdot)$ and using also the commutation formula
\[
\ \Delta\nabla_{i}f-\nabla_{i}\Delta f=\RRR_{is}\nabla^s f,
\]
we can compute
\begin{align*}
&(div_f(2(\Ric+\mathrm{Hess}(f)-\lambda_S g)))_i\\
=&e^{f}g^{kj}\nabla_k[2(\RRR_{ij}+\nabla^2_{ij}f-\lambda_S g_{ij})e^{-f}]\\
=&2e^{f}g^{kj}(\nabla_k\RRR_{ij}+\nabla_k\nabla^2_{ij}f)e^{-f}-2e^{f}[(\RRR_{ij}+\nabla^2_{ij}f-\lambda_S g_{ij})g^{jk}\nabla_k f]e^{-f}\\
=&(\nabla_i\RRR+2\Delta\nabla_i f)-2[(\RRR_{ij}+\nabla^2_{ij}f-\lambda_S g_{ij})g^{jk}\nabla_{k}f]\\
=&(\nabla_i R+2\nabla_i\Delta f+2\RRR_{is}\nabla^s f)-2[(\RRR_{ij}+\nabla^2_{ij}f-\lambda_S g_{ij})g^{jk}\nabla_k f]\\
=&(\nabla_i \RRR+2\nabla_i\Delta f-2g^{jk}\nabla^2_{ij} f\nabla_k f+2\lambda_S\nabla_i f)\\
=&\nabla_i (\RRR+2\Delta f -|\nabla f|^2+2\lambda_S f)\\
=&0.
\end{align*}
Using this and the fact that from the soliton equation \eqref{RS} we know that
\[
\ \nabla_l{X_k}+\nabla_{k}X_{l}=-2\RRR_{lk}+ 2\lambda_{S}g_{lk},
\]
we compute
\begin{align*}
&div_{f}[i_{\nabla f- X}(\Ric+\mathrm{Hess}(f)-\lambda_S g)]\\
=&e^{f}g^{li}\nabla_{l}[(\nabla_{k}f-X_{k})g^{kj}(\RRR_{ij}+\nabla^2_{ij}f-\lambda_S g_{ij})e^{-f}]\\
=&(\nabla^2_{lk}f-\nabla_lX_k)g^{kj}g^{li}(\RRR_{ij}+\nabla^2_{ij}f-\lambda_S g_{ij})\\
=&\frac{1}{2}[(2\nabla^2_{lk}f-\nabla_l X_{k}-\nabla_k X_l)g^{kj}g^{li}(\RRR_{ij}+\nabla^2_{ij}f-\lambda_S g_{ij})]\\
=&\frac{1}{2}[(2\nabla^2_{lk}f+2\RRR_{lk}-2\lambda_S g_{lk})g^{kj}g^{li}(\RRR_{ij}+\nabla^2_{ij}f-\lambda_S g_{ij})]\\
=&|\RRR_{ij}+\nabla^2_{ij}f-\lambda_S g_{ij}|^2,
\end{align*}
where in the third equality we have substituted $2\nabla_l X_{k}$ with $\nabla_l X_{k}+\nabla_k X_{l}$, since the skew-symmetric part of $\nabla X$ vanishes once we contract it with $\Ric+\mathrm{Hess}(f)-\lambda_{S} g$.
Hence we can conclude that
\[
\ 0\leq |\Ric+ \mathrm{Hess}(f)-\lambda_S g|^2 =div_f T
\]
for the one form 
\[
T=i_{\nabla f- X}(\Ric+\mathrm{Hess}(f)-\lambda_S g),
\]
and the theorem is proved provided that $\int_{M}div_{f}Te^{-f}=0$.

To this end, we are going to apply a sort of Karp's version of Stokes theorem on complete manifolds with density; \cite{Karp}. We start considering a sequence of Hessian cut-off $\{\chi_n\}$ on $M$, whose existence is assured by Corollary \ref{coro-hess}. 
Recalling the expression for $f$, we have thus to check that
\begin{equation}\label{Tv}
\int_M v^2 \left(\Ric-2\frac{\mathrm{Hess} (v)}{v}+2\frac{dv\otimes dv}{v^2}-\lambda_S g\right)\left(\frac{\nabla v}{v} - X,\nabla \chi_n \right) d\mathrm{vol} \to 0,\quad\textrm{as }n\to\infty.
\end{equation}

Beforehand, let us note for later purposes that by Theorem \ref{th_SolGrowth} and \eqref{ExpDec}, fixed a reference point $o\in M$, there exists a positive constant $b>0$ such that 
$e^{bd(x,o)^2}|Xv|(x)\to 0$ as $d(x,o)\to\infty$. In particular $|Xv|\in L^\infty (M)$ and, using also Bishop-Gromov theorem, $|Xv|\in L^2(M)$. 
%
In the following, given $g,h:M\to \mathbb{R}$, $g\lesssim f$ means that there exists a positive constant $C$ such that $|f(x)|\leq C|g(x)|$ for all $x\in M$. Moreover integrals are meant with respect to $d\mathrm{vol}$ unless otherwise specified.

To prove \eqref{Tv}, we write the integral as a sum of four different terms, which will be dealt with separately in the following.

\textbf{a)} Firstly, since $(\Ric - \lambda_{S} g)$ is bounded,
\begin{align*}
&\left|\int_M v^2 \left(\Ric-\lambda_S g\right)\left(\frac{\nabla v}{v} - X,\nabla \chi_n \right)\,d\mathrm{vol}\right|\\
&\lesssim \left\|v\nabla v\right\|_{L^1(\operatorname {supp} \nabla\chi_n)}+\left\|Xv^2\right\|_{L^1(\operatorname {supp} \nabla\chi_n)}\\
&\lesssim \left\|v\right\|_{L^2(\operatorname {supp} \nabla\chi_n)}\left\|\nabla v\right\|_{L^2(\operatorname {supp} \nabla\chi_n)}+\left\|Xv\right\|_{L^\infty(\operatorname {supp} \nabla\chi_n)}\left\|v\right\|_{L^1(\operatorname {supp} \nabla\chi_n)}\to 0,
\end{align*}
as $n\to\infty$.

\textbf{b)} Secondly, 
\begin{equation*}
\left|\int_M v^2 \left(\frac{\mathrm{Hess} (v)}{v}\right)\left(\frac{\nabla v}{v} - X,\nabla \chi_n \right)\,d\mathrm{vol}\right|
\lesssim (\left\|\nabla v\right\|_{L^2(\operatorname {supp} \nabla\chi_n)}+\left\|X v\right\|_{L^2(\operatorname {supp} \nabla\chi_n)})\|\mathrm{Hess} (v)\|_{L^2}\to 0,
\end{equation*}
as $n\to\infty$, provided one can prove that $\mathrm{Hess} (v)\in L^2(M)$. According to a global Calder\'on-Zygmund inequality, \cite[Theorem C]{GP}, this is verified if $v\in L^2(M)$ and $\Delta v\in L^2(M)$. This latter is true since, by equation \eqref{EL1},
\[
\|\Delta v\|_{L^2}\lesssim  \|v\|_{L^2}+\|v\ln v\|_{L^2},
\]
and, using also that $|\ln t|\leq \sqrt{t}+\sqrt{1/t}$ for all $t>0$,
\[
v^2\ln^2 v \leq v^3+2 v^2 + v \in L^1(M,d\mathrm{vol}).
\]

\textbf{c)} Similarly, 
\begin{equation}\label{c}
\left|\int_M v^2 \left(\frac{dv\otimes dv}{v^2}\right)\left(X,\nabla \chi_n \right)\,d\mathrm{vol}\right|\lesssim (\left\|\nabla v\right\|_{L^2(\operatorname {supp} \nabla\chi_n)}\left\|X \right\|_{L^\infty(\operatorname {supp} \nabla\chi_n)}).
\end{equation}
Let $h$ be as in Proposition \ref{DistLike} and let $\alpha(n)\in[n-1,n]$ be such that $ \left\{h=\alpha(n)\right\}$ is a regular hypersurface. By \eqref{ExpDec}, the coarea formula and Bishop-Gromov's inequality,
\begin{align}\label{L2decay}
\left(\int_{B_{n-1}\setminus B_{n-1-C_{m,K}}(o)}v^2\,d\mathrm{vol}\right)^{1/2}\leq\left(\int_{n-1-C_{m,K}}^{n-1}  \mathcal H_{m-1}(\partial B_{t}(o))\sup_{\partial B_{t}(o)}v^2dt\right)^{1/2} \lesssim e^{-n^\beta} 
\end{align}
for all $\beta<2$, with $\mathcal H_{m-1}$ denoting the $(m-1)$-dimensional Hausdorff measure of $\partial B_{t}(o)$. Accordingly,
\begin{align*}
\int_{B_{n-1}\setminus B_{n-1-C_{m,K}}(o)}|v\nabla v|\,d\mathrm{vol}\leq 
\left(\int_{B_{n-1}\setminus B_{n-1-C_{m,K}}(o)} v^2\,d\mathrm{vol}\right)^{1/2}\|\nabla v\|_{L^2(M)}\lesssim e^{-n^\beta}.
\end{align*}
Since $\nabla h$ is bounded, using again the coarea formula we can choose the $\alpha(n)$'s so that
$$
\int_{\{h= \alpha(n)\}} |v \nabla v|d\mathcal H_{m-1} \lesssim  e^{-n^{\beta'}}
$$
as $n\to\infty$ for all $\beta'<\beta$, and in particular for some $\beta'>1$. This latter estimate, \eqref{L2decay} and the $L^2$ integrability of $\Delta v$ proved above imply
\begin{align*}
\left\|\nabla v\right\|_{L^2(\operatorname {supp} \nabla\chi_n)}
\leq& \int_{\{h>\alpha(n)\}}\left\langle\nabla v,\nabla v\right\rangle\,d\mathrm{vol} \\
=& \int_{\{h=\alpha(n)\}} v\left\langle\nabla v,-\frac{\nabla  h}{|\nabla  h|}\right\rangle d\mathcal{H}_{m-1} -\int_{\{h>\alpha(n)\}} v\Delta v\,d\mathrm{vol}\\
\leq& \int_{\{h=\alpha(n)\}} |v \nabla v|d\mathcal{H}_{m-1}
+ \left(\int_{\{h>\alpha(n)\}} v^2\,d\mathrm{vol}\right)^{1/2}\left(\int_{\{h>\alpha(n)\}} (\Delta v)^2\,d\mathrm{vol}\right)^{1/2}\\
\lesssim& e^{-n^{\beta'}}.
\end{align*}
Inserting in \eqref{c} and applying Theorem \ref{th_SolGrowth} give
\begin{equation*}
\left|\int_M v^2 \left(\frac{dv\otimes dv}{v^2}\right)\left(X,\nabla \chi_n \right)\,d\mathrm{vol}\right|\to 0,
\end{equation*}
as $n\to\infty$.

\textbf{d)} Finally, 
\begin{align*}
\int_M\left\langle\nabla v,\frac{\nabla v}{v}\right\rangle\left\langle\nabla v,\nabla \chi_n\right\rangle\,d\mathrm{vol}
=&\int_M\left\langle\nabla v,\nabla \ln v\right\rangle\left\langle\nabla v,\nabla \chi_n\right\rangle\,d\mathrm{vol}\\
=&-\int_M\ln v \Delta v\left\langle\nabla v,\nabla \chi_n\right\rangle\,d\mathrm{vol}-\int_M\ln v\left\langle\nabla v,\nabla\left\langle\nabla v,\nabla \chi_n\right\rangle\right\rangle\,d\mathrm{vol}.
\end{align*}
On the one hand, by \eqref{EL1},
\begin{align*}
\|\ln v\Delta v\|_{L^2}\lesssim \|v\ln v\|_{L^2}+\|v\ln^2 v\|_{L^2},
\end{align*}
and, using that $|\ln x|\leq 2\sqrt[4]{x}+2\sqrt[4]{1/x}$ for all $x>0$, we know that
\[
v^2\ln^4 v \lesssim (v^3+v^{5/2}+v^2 +v^{3/2}+ v) \in L^1(M,d\mathrm{vol}).
\]
Then 
\begin{align*}
\left|\int_M\ln v \Delta v\left\langle\nabla v,\nabla \chi_n\right\rangle\,d\mathrm{vol}\right|
&\leq \|\ln v\Delta v\|_{L^2}\|\nabla v\|_{L^2}\|\nabla\chi_n\|_{L^\infty}\to 0,
\end{align*}
as $n\to \infty$.
On the other hand,
\begin{align*}
\int_M\ln v\left\langle\nabla v,\nabla\left\langle\nabla v,\nabla \chi_n\right\rangle\right\rangle\,d\mathrm{vol}
=&\int_M\ln v\left[\left\langle\nabla_{\nabla v}\nabla v,\nabla \chi_n\right\rangle+
\left\langle\nabla v,\nabla_{\nabla v}\nabla \chi_n\right\rangle\right]\,d\mathrm{vol}\\
=&\int_M \ln v \mathrm{Hess}(v)(\nabla v,\nabla \chi_n)\,d\mathrm{vol}+\int_M \ln v \mathrm{Hess}(\chi_n)(\nabla v,\nabla v)\,d\mathrm{vol}.
\end{align*}
Since $\|\nabla\chi_n\|_{L^\infty}\to 0$ and, as proved above, $\mathrm{Hess} (v)\in L^2(M)$, we have that
\[
\left|\int_M \ln v \mathrm{Hess}(v)(\nabla v,\nabla \chi_n)\,d\mathrm{vol}\right|\to 0
\]
as $n\to \infty$ provided that 
\begin{equation}\label{lnvnablav}
(\ln v) \nabla v\in L^2(M).
\end{equation}
This is true since
\begin{align*}
\int_M \ln^2 v |\nabla v|^2\,d\mathrm{vol}
=& \int_M \left\langle\nabla v,\ln^2 v \nabla v\right\rangle\,d\mathrm{vol}\\
=&\int_M \left\langle\nabla v, \nabla (v\ln^2 v- 2v\ln v +2v)\right\rangle\,d\mathrm{vol}\\
=&-\int_M \Delta v (v\ln^2 v- 2v\ln v +2v)\,d\mathrm{vol}\\
\leq& \|\Delta v\|_{L^2}\|v\ln^2 v- 2v\ln v +2v\|_{L^2}<\infty,
\end{align*}
where the integration by part can be justified using the equation as above. Moreover, since $\|\mathrm{Hess}\chi_n\|_{L^\infty}\to 0$, 
we have that
\[
\left|\int_M \ln v \mathrm{Hess}(\chi_n)(\nabla v,\nabla v)\,d\mathrm{vol}\right|\to 0
\]
as $n\to \infty$ provided that $\ln v |\nabla v|^2\in L^1(M)$. This is true by \eqref{lnvnablav} and the fact that, thanks to the exponential decay of $v$, $|\ln v|\lesssim |\ln v|^2$.
\end{proof}

\begin{acknowledgement*}
This work is supported by a public grant overseen by the French National Research Agency (ANR) as part of the ``Investissements d'Avenir'' program (reference: ANR-10-LABX-0098). The authors are members of the ``Gruppo Nazionale per l'Analisi Matematica, la Probabilit\'a e le loro Applicazioni'' (GNAMPA) of the Istituto Nazionale di Alta Matematica (INdAM). 
We would like to thank Luciano Mari for comments on a previous version of the paper.
\end{acknowledgement*}

\bibliographystyle{amsplain}
\bibliography{ExtrLSI}
 
\end{document}